\documentclass{amsart}

\newtheorem{theorem}{Theorem}[section]

\theoremstyle{definition}

\newtheorem{example}[theorem]{Example}

\theoremstyle{remark}
\newtheorem{remark}[theorem]{Remark}

\numberwithin{equation}{section}

\begin{document}

\title{Real affine varieties of nonnegative curvature}


\author{Xiaoyang Chen}
\address{Department of Mathematics, University of Macau, Macau, China}
\email{xychen100@gmail.com}
\thanks{ The author was partially supported by FDCT/016/2013/A1.}

\subjclass[2010]{Primary 53C20, 53C24; Secondary 53C21, 53C25}


\keywords{real affine variety, nonnegative curvature, soul theorem}

\begin{abstract}
Let $X_\mathbb{C}$ be a smooth real affine variety with compact real points $X_\mathbb{R}$. We show that $X_\mathbb{C}$ is diffeomorphic to the normal bundle of $X_\mathbb{R}$
  provided that $X_\mathbb{C}$ admits a complete Riemannian metric  of nonnegative sectional curvature which is also invariant under the conjugation.
\end{abstract}

\maketitle

\section{Introduction}

In \cite{Tot03}, Totaro asked whether a compact manifold $M$ with nonnegative sectional curvature has a good complexification, i.e. there exists a smooth real affine variety
$X_\mathbb{C}$ such that $M$ is diffeomorphic to $X_\mathbb{R}$ ( the real points of $X_\mathbb{C}$ ) and the inclusion $ X_\mathbb{R} \to X_\mathbb{C}$ is a homotopy equivalence.
 A positive anwser to Totaro's question would in particular resolves an odd problem of Hopf
 saying that a compact nonnegatively curved manifold  has nonnegative Euler number \cite{Kul78}.
Recently I. Biswas and M. Mj has obtained a positive answer to Totaro's question for three dimensional manifolds \cite{BM}.

According to a remarkable theorem of Nash and Tognoli \cite{Nash} \cite{Tog}, any compact differentiable manifold can be realized as the real points $X_\mathbb{R}$  of a smooth real affine variety
$X_\mathbb{C}$. Now Totaro's question can be reformatted as the following:

\vskip0.2in

Let  $X_\mathbb{C}$ be a smooth real affine variety with compact real points $X_\mathbb{R}$. If $X_\mathbb{R} $ admits a complete Riemannian metric  of nonnegative sectional curvature, does there exist
a smooth real affine variety $U_\mathbb{C}$ such that $X_\mathbb{R}$ is diffeomorphic to $U_\mathbb{R}$ and the inclusion $U_\mathbb{R} \to U_\mathbb{C}$ is a homotopy equivalence?

\vskip0.2in

 In this paper we give a positive answer to the above question under a stronger assumption.
 More precisely, we are going to prove the following

\begin{theorem} \label{main}
Let $X_\mathbb{C}$ be a smooth real affine variety with compact real points $X_\mathbb{R}$. If $X_\mathbb{C}$ admits a complete Riemannian metric $g$ of nonnegative sectional curvature such that
$\tau^*g=g$, where $\tau: X_\mathbb{C} \to X_\mathbb{C}$ is the conjugation,
 then $X_\mathbb{C}$ is diffeomorphic to the normal bundle of $X_\mathbb{R}$.
 In particular, the inclusion  $ X_\mathbb{R} \to X_\mathbb{C}$ is a homotopy equivalence.

\end{theorem}

Note that in Theorem \ref{main},  $X_\mathbb{R}$ is the fixed point set of the isometry $\tau$. It follows that $X_\mathbb{R}$ is a totally geodesic submanifold of $X_\mathbb{C}$,
in particular,  the induced metric of $g$ has nonnegative sectional curvature on $X_\mathbb{R}$.

\begin{remark}
The assumption that $X_\mathbb{R}$ is compact in Theorem \ref{main} can not be dropped.
An example is given by $X_\mathbb{C}=\{{(z_1,z_2) \in \mathbb{C}^2|z_1^2-z_2^2=1}\}$.
Define a map $F$ by
$$
F:  X_{\mathbb{C}} \rightarrow \mathbf{S}^1 \times \mathbb{R}, \ (z_1, z_2) \mapsto (\frac{z_1+z_2}{|z_1 + z_2|}, ln |z_1-z_2|).
$$
Then $F$ is a diffeomorphism.
By pulling back the product metric on $\mathbf{S}^1 \times \mathbb{R}$, we see that $X_\mathbb{C}$ admits a complete flat Riemannian metric which is also invariant under the conjugation. However, the inclusion $X_{\mathbb{R}} \rightarrow X_{\mathbb{C}}$ is $not$ a homotopy equivalence since $X_{\mathbb{C}}$ is connected but $X_\mathbb{R}$ has two noncompact  connected components.

\end{remark}

\begin{example}

Let $G$ be a closed subgroup of $U(n)$ and $G_\mathbb{C}$ be the complexification of $G$. 
It's well known that $G_\mathbb{C}$ is a real affine variety \cite{Kul78}. Note that there is  a diffeomorphism $F: G \times Lie(G) \to G_\mathbb{C}$ given by

$$F: G \times Lie (G) \to G_\mathbb{C}, \; (x,A) \mapsto   e^{iA} x,$$
where $Lie(G)$ is the Lie algebra of $G$.

Let $\tau: G_\mathbb{C} \to G_\mathbb{C}$ be the conjugation and $\phi: G\times Lie(G) \to G \times Lie(G), (x, A) \mapsto (\bar{x}, -\bar{A})$ be an involution.
Then it is direct to check that $\tau F= F \phi$. Let $h$ be a bi-invariant metric on $G$ which is also invariant under the conjugation.
By pulling back the product  of $h$ and the flat metric on $ Lie(G)$, we get 
 a complete Riemannian metric $g$ on $G_\mathbb{C}$ with nonnegative sectional curvature such that $\tau^*g=g$. However, in general $g$ is $not$ $G_\mathbb{C}$ invariant as
$F$ is $not$ a group homomorphism in general.

\end{example}

\noindent \textbf {Acknowledgement}:  The  author would like to express his gratitude to Professors Huai-Dong Cao and Karsten Grove for support.

\section{Proof of Theorem \ref{main}}

The proof of Theorem \ref{main} is based on the soul construction of Cheeger and Gromoll in \cite{CG72}. In fact, we are going to show that $X_\mathbb{R}$ is a soul of $X_\mathbb{C}$.

We first prove the following theorem which is interesting on its own.

\begin{theorem} \label{fix}
Let $(E,g)$ be a complete noncompact connected manifold with nonnegative sectional curvature and $f$ be an isometry of $(E,g)$ with $E^f$  compact, where $E^f$ is the fixed point set of $f$.
Then $E^f$ is contained in a soul of $E$.
\end{theorem}

\begin{proof}
Fix $x_0 \in E^f$ and define $B(x)=sup_{\gamma} B_\gamma(x), x \in E$, where $\gamma: [0, +\infty) \to E$ is a ray with $\gamma(0)=x_0$ and $B_\gamma$ is the  Busemann function given by
$$B_\gamma(x)=\displaystyle{\lim_{t\to +\infty}}(t- d (x, \gamma(t)).$$
Since $f$ is an isometry with $f(x_0)=x_0$, we see $B_{f\circ \gamma}(f(x))=B_\gamma(x)$. Then we get  $B \circ f=B$ which will be crucial for us.
By the arguments in \cite{CG72}, we get a soul $S$ of $E$ which satisfies $f(S)=S$ (e.g, Corollary 6.3 in \cite{CG72}).
Moreover, there is a distance nonincreasing retraction $P: E \to S$ (e.g, \cite{Yim}). By checking the construction of $P$ carefully and using $B \circ f=B$,
we also see
$$P \circ f=f \circ P.$$
For any $x \in E^f$, we have $f(x)=x$ and
$$d(P(x), f \circ P ( x))=d(P(x), P \circ f(x)) \leq d(x, f(x))=0.$$
Then $P(x)=f \circ P (x)$ and hence
\begin{equation} \label{1}
P(E^f) \subseteq E^f.
\end{equation}
It follows that $P(E^f) \subseteq E^f \cap S$. On the other hand, since $P_{|S}=Id_{|S}$, we see $E^f \cap S= P(E^f \cap S) \subseteq P(E^f).$
Then
$$P(E^f)=E^f \cap S.$$
Let $H:E \times [0,1]\to E$ be the homotopy between $Id_{|E}$ and $P$ constructed in \cite{Yim}. Then for each $s\in [0,1]$, $H_s:=H(,s): E \to E$ is a distance nonincreasing map.
By checking the construction of $H$ carefully and using $B \circ f=B$, we see
$$H_s \circ f= f \circ H_s.$$
By a similar argument in the proof of \ref{1}, we see $H_s (E^f)\subseteq E^f$  and hence $H_{|E^f \times [0,1]}: E^f \times [0,1] \to E^f$ is a homotopy between $Id_{|E^f}$ and $P_{|E^f}$.
It follows that $E^f \cap S$ is a retraction of $E^f$. Since $E^f$ is a compact manifold without boundary,
we see $E^f \cap S =E^f$ and hence $E^f \subseteq S$.

\end{proof}

Now Theorem \ref{main} follows from Theorem \ref{fix} easily. Without loss of generality, we can assume that $X_\mathbb{C}$ is connected.
Then by Theorem \ref{fix}, $X_\mathbb{R}$ is contained in a soul $S$ of $X_\mathbb{C}$.
It's known that a smooth real affine variety $X_\mathbb{C}$ has the homotopy type of a $CW$ complex of dimension $\leq dim (X_\mathbb{R})$ (Theorem 7.2 in \cite{Milnor}).
It follows that $X_\mathbb{R}$ has the same dimension as $S$ and then $X_\mathbb{R}=S$ since $X_\mathbb{R}$ is a closed submanifold of $S$ and $S$ is connected.

\bibliographystyle{amsplain}

\end{document}